\documentclass{gtpart}     
\author{E. Mac\'{\i}as-Virg\'os}
\address{Departamento de Xeometr\'{i}a e Topolox\'{\i}a, Universidade de Santiago de Compostela, 15782-SPAIN}
\email{quique.macias@usc.es}
\urladdr{http://webspersoais.usc.es/persoais/quique.macias/}

\author{M. J. Pereira-S\'aez}
\address{Departamento de Econom\'{\i}a Aplicada II, Universidade de A Coru\~na, 15071-SPAIN}
\email{maria.jose.pereira@udc.es}
   \title[An upper bound for the LS category of $Sp(n)$]{An upper bound for the Lusternik-Shnirelmann category of the symplectic group}

\volumenumber{}
\issuenumber{}
\publicationyear{}
\papernumber{}
\startpage{}
\endpage{}
\doi{}
\MR{}
\Zbl{}
\received{}
\revised{}
\accepted{}
\published{}
\publishedonline{}
\proposed{}
\seconded{}
\corresponding{}
\editor{}
\version{}

\newtheorem{teor}{Theorem}[section]    
\newtheorem{coro}[teor]{Corollary} 
\theoremstyle{definition}
\newtheorem{ejem}[teor]{Example} 
\newtheorem*{note}{Remark}             
\newcommand{\cat}{\mathop{\mathrm{cat}}\nolimits}
\newcommand{\I}{\mathbf{i}}

\newcommand{\KK}{\mathbb{K}}
\newcommand{\RR}{\mathbb{R}}
\newcommand{\CC}{\mathbb{C}}
\newcommand{\HH}{\mathbb{H}}
\newcommand{\MM}{\mathcal{M}_{n\times n}(\KK)}
\newcommand{\Tr}{\mathrm{Tr}}
\newcommand{\diag}{\mathrm{diag}}
\newcommand{\grad}{\mathop{\mathrm{grad}}}

\newcommand{\bcat}{\mathop{\mathrm{bcat}}\nolimits}

\begin{document} 

\begin{abstract}    
We prove that the LS category of the symplectic group $Sp(n)$ is bounded above by $\binom{n+1}{2}$. This is achieved by computing the number of critical levels of a height function.

\end{abstract}

\maketitle

\section{Introduction}
The Lusternik-Shnirelmann category of a space $G$ is defined as the least integer $k\geq 0$ such that $G$ admits a covering by $k+1$ open subsets that are contractible in $G$ ({\em categorical} open sets). A classical problem raised by T.~Ganea \cite{Gan1971} is to compute this invariant for Lie groups.
However only a small number of calculations have been carried out, mainly due to the difficulty of the homotopical techniques involved. This is particularly true for the symplectic group where the only known results are $\cat Sp(2)=3$, obtained by P.~Schweitzer in 1965 \cite{Sch1965} and $\cat Sp(3)=5$  proved by L.~Fern\'andez, A.~G.~Tato, J.~Strom and D.~Tanr\'e  in 2001 \cite{FerGomStrTan2004} and by N.~Iwase and M.~Mimura \cite{IwaMim2004}. The latter authors also proved that $\cat Sp(n)\geq n+2$ for $n\geq 3$.

In 2009, M.~Hunziker and M.~R.~Sepanski \cite{HunSep2009} showed that the LS category of a simple, simply connected, compact Lie group $G$ is bounded above by the sum of the relative categories of the conjugacy classes ${\cal O}_k$ of $\exp v_k$, $0\leq k\leq n=\dim G$, where $\{v_0,\dots,v_n\}$ are the vertices of the fundamental alcove for the action of the affine Weyl group on the Lie algebra of a maximal torus of $G$. 
For $G=SU(n)$ the computation recovers W. Singhof's  result $\cat SU(n)=n-1$ \cite{Sing1975}.    For the symplectic group $G=Sp(n)$ their formula is
\begin{equation}\label{HS}
\cat Sp(n)+1\leq \sum_{k=0}^n{\left(\cat_{Sp(n)}{\cal O}_k +1\right)}.
\end{equation}
Moreover Hunzinker and Sepanski conjectured that $\cat_{Sp(n)}{\cal O}_k\leq \min\{k,n-k\}$, which should imply that
$$\cat Sp(n)\leq\lfloor\frac{(n+2)^2}{4}\rfloor-1.$$
The submanifolds ${\cal O}_k$ are Grassmannians, but the computation of their relative category seems to rely in a subtle way on the difference between left and right eigenvalues of quaternionic matrices, a subject where very little is known (see for instance the authors' results in \cite{MacPer2009, MacPer2010} about left eigenvalues of symplectic matrices).

In this paper we shall prove that  formula (\ref{HS}) is a very particular case of our Theorem \ref{COMPS}, which compares the category of the ambient manifold $G$ with the category of the critical subset $\Sigma(h_X)$ of any height function $h_X$. As a consequence we obtain the upper bound
$$\cat Sp(n)\leq \binom{n+1}{2}.$$

\section{Height functions on Lie groups}
In \cite{GomMacPer2011} the authors studied arbitrary height functions on a Lie group $G$ of orthogonal type. Let us briefly recall some results.

Let $\KK$ be either $\RR$ (reals), $\CC$ (complex) or $\HH$ (quaternions). The orthogonal group (resp. unitary, symplectic)
$$G=O(n,\KK)=\{A\in \MM\colon AA^*=I\}$$
is embedded in the euclidean space $\MM$, where the usual euclidean product is given by  $\langle A,B\rangle=\Re\Tr(A^*B)$ (we denote $\Re\Tr$ the real part of the trace and $*$ the conjugate transpose). Consequently, the height function $h_X\colon G \to \RR$ with respect to the hyperplane perpendicular to the vector $X\in \MM$, with $\vert X \vert =1$, is given by the formula
$$h_X(A)=\Re\Tr(X^*A).$$
Depending on the singular values of $X$ this function will be Morse (non-degenerate critical points) or just Bott-Morse (non-degenerate critical {\em submanifolds}). In fact, we have the following structure theorem for the critical $\Sigma(h_X)$ set of $h_X$. 

For $p\geq 1$ let us denote $\Sigma(p)$ the disjoint union $G_0^p\sqcup \cdots \sqcup G_p^p$ of  the Grassmannians 
$$G_q^p=\frac{O(p)}{O(q)\times O(p-q)}.$$

\begin{teor}\label{ESTRUCT}
 Let $0< t_1<\dots<t_k$ be the singular values of $X$, with multiplicities $n_0, n_1,\dots, n_k$. Then
$$\Sigma(h_X)\cong O(n_o)\times \Sigma(n_1)\times \cdots \times \Sigma(n_k).$$ \end{teor}

\begin{proof} Let $X=UDV^*$ be the singular value decomposition (SVD) of $X$, with $$D=\diag(0,\stackrel{n_0)}{\dots},0,t_1,\stackrel{n_1)}{\dots},t_1,\dots,t_k,\stackrel{n_k)}{\dots}, t_k).$$  Then the gradient at the point $A\in G$ equals
$$(\grad h_X)_A=\frac{1}{2}(X^*-AXA)=V(\grad h_D)_{V^*AU}U^*$$ while the Hessian $ T_AG \to T_AG$ verifies
$$(Hh_X)_A(Y)=(-1/2)(AXY+YXA)=V(Hh_D)_{V^*AU}(V^*YU)U^*.$$
Hence it suffices to study the functions  $h_X$ with $X={t_iI}$, whose critical set  is $\Sigma(n_i)$. The function $h_I$ was first studied by T.~Frankel \cite{Fran1963}.
\end{proof}

\begin{ejem}\label{U2}
 {\rm Let $G=U(2)$ and $X=\diag(0,1)$, that is $n_0=1$, $n_1=1$. Then $\Sigma(h_X)\cong U(1)\times \Sigma(1)=S^1\sqcup S^1$, two disjoint circles. On the other hand, if $X=\diag(1,1)$ then $n_0=0$ and $n_1=2$ hence $\Sigma(h_X)=\Sigma(2)$, the union of two points and a sphere $S^3/S^1\cong S^2$. Finally, when $X=\diag(1,2)$, that is $n_0=0$, $n_1=1$, $n_2=1$, the critical set $\Sigma(1)\times \Sigma(1)$ are four points.
}\end{ejem}

 \section{LS category and critical points}
The {\em relative} LS category of a submanifold $A\subset G$ is the minimum number $k$ such that $A$ can be covered by $k+1$ categorical open sets of $G$. Clearly if $A,B\subset G$ are disjoint compact submanifolds  then
\begin{equation}\label{UNION}
\cat_G(A\cup B)\leq \cat_GA +\cat_GB +1.
\end{equation}
For $G$ a closed manifold it is well known that $\cat G+1$ is a lower bound for the number of critical points of any smooth function $f\colon G \to \RR$. When the critical set $\Sigma(f)$ is not finite we still have the following formula \cite{RudSch2003,Ree1972}:
\begin{teor}\label{RUDYAK} Let $c_1,\dots,c_p$ be the critical values of the function $f$. Then$$\cat G+1 \leq \sum_{q=1}^p{\left(\cat_G\left(\Sigma(f)\cap f^{-1}(c_q)\right)+1\right)}.$$
\end{teor}
Taking into account that in a path-connected manifold a finite collection of points has category $0$ we have:
\begin{coro}\label{LEVELS}
 Assume $G$ is connected. When all the critical points are isolated, $\cat G+1$ is bounded above by the number $p$ of critical {\em levels} of $f$.
\end{coro}

We start by proving in an easy way the result of Hunzinker and Sepanski cited in the Introduction.

\begin{ejem} {\rm Let $G=Sp(n)$ be the symplectic (quaternionic) group and consider the height function $h_I$, which is invariant by the adjoint action.  The critical set is $\Sigma(n)=G_0^n\sqcup \cdots \sqcup G_n^n$, where the Grassmannian $G_q^n={\cal O}_q$ is the orbit of the matrix $\diag(-I_q,+I_{n-q})$. Then, accordingly to Theorem \ref{RUDYAK},
$$\cat Sp(n)+1 \leq \sum_{q=0}^n{\left(\cat_G G_q^n +1\right)}$$
which is exactly formula (\ref{HS}).
}\end{ejem}

In the general case of an arbitrary function $h_X$ on the Lie group $G$ the components of the critical set are the products
$$\Sigma[n_0,i_1,\dots,i_k]=O(n_0)\times G_{i_1}^{n_1}\times \cdots \times G_{i_k}^{n_k}, \quad  0\leq i_q \leq n_q.$$
By applying Theorem \ref{RUDYAK} and formula (\ref{UNION}) we obtain

\begin{teor} \label{COMPS}
$$\cat G +1 \leq \sum{\left(\cat_G \Sigma[n_0,i_1,\dots,i_k] +1\right)}.$$
\end{teor}

\begin{ejem}
Let $G=U(2)$ and $h_X$ as in Example \ref{U2} with $X=I$. Then
$$\cat U(2)+1 \leq 2(\cat_G(*) +1)+(\cat_G G_1^2+1)=3$$
because $\cat_{U(2)}G_1^2=0$. Indeed, although $\cat S^2=1$ it happens that $G_1^2$, which is the orbit by the adjoint action of the matrix $\diag(1,-1)$,  is contained in the open set 
$$\Omega_G(\I)=\{A\in G \colon \det(A-\I I)\neq 0\}$$ 
which is contractible \cite{GomMacPer2011}.
(It is known that $\cat U(n)=n$).
\end{ejem}

\begin{ejem} {\rm
Let $G=Sp(n)$ and take $n_0=n-1$. Then
$$\cat Sp(n)  \leq 2 \cat_{Sp(n)}Sp(n-1)+1.$$
Note that $\Sigma(h_X)=Sp(n-1)\times \Sigma(1)$ and that the two copies of $Sp(n-1)$ lie in two different critical levels $h_X^{-1}(\pm t_1)$.
}\end{ejem}

In order to improve Theorem \ref{COMPS} we should study the critical levels of an arbitrary height function $h_X$. The relative category is invariant by diffeomorphisms, then, accordingly to Theorem \ref{ESTRUCT} we can suppose that $X=\diag(0,\dots,t_1,\dots,t_k)$. Each component $\Sigma[n_0,i_1,\dots,i_k]$ of the critical set is formed by the matrices 
$$A=\begin{pmatrix}
A_0&\cr
&A_1\cr
&&\dots\cr
&&&A_k\cr
\end{pmatrix}$$ where $A_q\in O(n_q)$, $0\leq q\leq k$, and
$A_q^2=I$ for $q\geq 1$. Moreover each $A_q$, $q\geq 1$, diagonalizes to the real matrix 
$$\begin{pmatrix}
-I_{i_q}&\cr
&+I_{n_q-i_q}\cr
\end{pmatrix}.$$
Then
$$h_X(A)=\Re\Tr(X^*A)=t_1(n_1-2i_1)+\cdots+t_k(n_k-2i_k).$$
In particular, when $h_X$ is a Morse function (that is $n_0=0$, $n_1=\dots =n_k=1$) we have
that the possible critical values are
$$t_1\varepsilon_1+\cdots+t_k\varepsilon_k.$$

\begin{coro}\label{BOUND}
$$\cat Sp(n)\leq \binom{n+1}{2}.$$
\end{coro}

\begin{proof}
Let us consider the Morse function $h_X$ with $X=\diag(1,2,\dots,n)$. In each critical level there is a matrix of the form $A=\diag(\varepsilon_1, \dots,\varepsilon_n)$ with $\varepsilon_i=\pm 1$. Then we observe that
\begin{enumerate}
\item
The maximum value is $C=1+\cdots+n=\binom{n+1}{2}$.
\item
If $c$ is a critical value then $-c$ is a critical value too.
\item
Changing $+1$ by $-1$ we observe that the critical values
are either $$-C, -C+2,\dots,-2, 0, 2, \dots, C$$ or $$-C,-C+2,\dots,-3, -1,1,3, \dots, C.$$

In both cases the number of critical levels is
$\binom{n+1}{2}+1$. Then the result follows from Corollary \ref{LEVELS}.
\end{enumerate}
\end{proof}

 \begin{note} An explicit categorical covering  of $G=Sp(2)$ is formed by the four open sets $\Omega_G(\pm I)$, $\Omega(\pm P)$, where $P=\diag(-1,1)$ and $\Omega_G(A)=\{X\in G\colon \det(A+X)\neq 0\}$.
\end{note}

\begin{note}\label{IWASE} Theorem \ref{RUDYAK}  is proved in \cite{RudSch2003} for several generalizations of LS category, including the so-called {\em ball category} $\bcat$ (coverings by smoothly embedded balls \cite{Sin1979}). Taking into account that for the function given in Corollary \ref{BOUND}  the critical points of each critical level are contained in the ball $\Omega_G(\I)$ we actually obtain $\bcat(Sp(n)\leq \binom{n+1}{2}$. \end{note}

%
%
%
%

\end{document}